\documentclass[11pt]{article}
\usepackage[utf8]{inputenc}
\usepackage[T1]{fontenc}
\usepackage{amsmath, amssymb, amsthm}
\usepackage{geometry}
\usepackage{hyperref}

\newcommand{\R}{\mathbb{R}}
\newcommand{\N}{\mathbb{N}}
\newcommand{\Rp}{\R^n_+}
\newcommand{\pt}{\partial}
\newcommand{\eps}{\varepsilon}
\DeclareMathOperator{\dv}{div}

\newtheorem{theorem}{Theorem}
\newtheorem{corollary}{Corollary}
\newtheorem{lemma}{Lemma}
\theoremstyle{remark}
\newtheorem{remark}{Remark}
\newtheorem{example}{Example}
\newtheorem{definition}{Definition}

\title{Apriori Estimates and Nonexistence Results for Some Quasilinear Problems in a Half-Space}
\author{Evgeny Galakhov \and Olga Salieva}
\date{\today}

\begin{document}
\maketitle

\section*{Introduction}

Whether a nonlinear elliptic equation admits a positive solution on an unbounded
domain is an old question, and the answer is often that it does not. Statements of
this kind go back to Gidas and Spruck \cite{GS81} and are usually called
Liouville-type theorems. They matter for more than their own sake: a priori bounds
for solutions on bounded domains, and with them existence theory through degree or
blow-up methods, typically rest on knowing that the limiting problem on $\R^n$ or on
a half-space has no nontrivial solution. A nonexistence theorem in the limit is what
lets a rescaling argument convert local information near a bad point into a uniform
bound, so the two questions are really one.

The half-space $\Rp=\{x_n>0\}$ sits between the whole space and a bounded domain,
and it is the natural setting for the behaviour of a solution near a boundary point.
If one rescales a problem posed on a smooth bounded domain around such a point, the
limiting equation lives on $\Rp$ and carries a Dirichlet condition on the flat
boundary, while the rescaled solutions keep a uniform Lipschitz bound. This is the
geometry we work in, and it is why we ask throughout for a bounded gradient. The same
class of problems arises in the modelling of nonlinear diffusion and non-Newtonian
fluid flow, where $-\Delta_p$ replaces the Laplacian and the exponent $p$ records how
the flux responds to the gradient.

For the model Dirichlet problem
\[
-\Delta_p u = u^q \ \text{ in } \Rp,\qquad u=0 \ \text{ on } \pt\Rp,
\]
the answer depends on $p$, $q$ and the dimension $n$ in a fairly delicate way. When
$p=2$ the equation is of Lane--Emden type, and the moving-plane method used by
Dancer \cite{Dancer} and its later refinements \cite{ChenLinZou} settled nonexistence
over wide ranges of $q$; Dupaigne, Sirakov and Souplet \cite{DSS} eventually closed
the semilinear case altogether, excluding positive solutions for every $q>1$. The
quasilinear case $p\neq2$ is more resistant, mainly because the moving-plane and
Kelvin-transform arguments that carry the semilinear theory do not transfer in the
same shape: the $p$-Laplacian is neither linear nor conformally covariant, and its
solutions are in general no better than $C^{1,\alpha}_{loc}$. Farina, Montoro, Riey
and Sciunzi \cite{FMRS,FMS,FMS2} worked around this by first proving that solutions
are monotone in $x_n$, and then, through a delicate test-function analysis, obtaining
nonexistence below an explicit critical exponent $q_c(n,p)$. Their bounds are sharp
in what they cover, but they are silent once $q$ reaches $q_c$, and this is the range
we want to enter.

Our route to it is deliberately elementary, and the method is the point of the paper
as much as the results are. Rather than the sharp integral machinery of Mitidieri and
Pohozaev, we use a single, self-improving a priori bound built from one structural
hypothesis on the operator: weak $p$-coercivity, in the sense isolated by D'Ambrosio
and Mitidieri \cite{DM0}. This asks only that
\[
{\cal A}(x,t,\xi)\cdot\xi\ \ge\ k\,|{\cal A}(x,t,\xi)|^{p'},\qquad p'=\tfrac{p}{p-1},
\]
with no upper growth condition on ${\cal A}$ and no regularity in $(t,\xi)$ beyond the
Carath\'eodory property. The condition is mild: the $p$-Laplacian satisfies it with
$k=1$, and so do the mean-curvature operator and a wide range of degenerate operators.
Its force is that it feeds on itself. Along any supersolution, weak $p$-coercivity
combines with Cauchy--Schwarz and the Lipschitz bound $|\nabla u|\le c_1$ to pin the
flux down pointwise,
\[
|{\cal A}(x,u,\nabla u)|\ \le\ k^{1-p}c_1^{p-1}\qquad\text{a.e. in }\Rp,
\]
without ever bounding ${\cal A}$ from above by hand. This is the self-improvement
observed in \cite{DM0}, and it is the whole engine of the argument.

From there the estimate follows from a single integration by parts. Testing the
supersolution inequality against one radial Lipschitz cutoff on a ball $B_{2R}$ and
inserting the pointwise flux bound turns the local mass of $f(u)$ into a capacity term
of order $R^{n-1}$; comparing this with the trivial lower bound $f(\inf_{B_R}u)\,
\omega_nR^n$ and passing to the generalized inverse of $f$ controls $\inf_{B_R}u$, and
the Lipschitz bound upgrades the infimum to a supremum. The radius $R$ is still free,
and balancing it against the nonlinearity produces an $L^\infty$ estimate for $u$ in
terms of its Lipschitz constant alone (Theorem~\ref{corr:thm1}). For $f(t)=t^q$ the
balanced exponent is $p/(q+1)$, exactly the number forced by the scaling
$u_\lambda(x)=\lambda^{p/(q-p+1)}u(\lambda x)$.

It is worth being clear about what this trades against the capacity method of
\cite{MP,BVP}. The integral estimate $\int_{B_R}u^q\le cR^{\,n-pq/(q-p+1)}$ used there
is genuinely sharper than our capacity bound, and we do not recover it. What we gain
is reach and transparency: the whole estimate rests on one integration by parts, it
never touches the structure of ${\cal A}$ beyond weak $p$-coercivity, and after
balancing $R$ it lands on the same critical exponent $p/(q+1)$ regardless. The same
argument therefore covers the mean-curvature operator and exponential nonlinearities
such as $f(t)=e^t-1$ with no change (Examples~\ref{ex1} and \ref{ex2}), cases where
no polynomial structure is available to feed a scaling argument.

From that bound we obtain Theorem~\ref{thm2} and, together with the growth lemma
Lemma~\ref{lem:growth}, Theorem~\ref{thm:case2}. Theorem~\ref{thm2} extends the known nonexistence range from
$q_c(n,p)$ up to $q_c(n-1,p)$ without assuming the solution is bounded, which is new
in the gap between these two exponents. A companion result (Theorem~\ref{thm:case1}),
independent of the bound and resting instead on the capacity estimate, excludes solutions that stay bounded below along
a hyperplane for every $q>p-1$; as a corollary
(Corollary~\ref{cor2}) we get an unconditional nonexistence result when the
nonlinearity is at least $(p-1)$-homogeneous near the origin. The remaining regime is
handled through a growth lemma (Lemma~\ref{lem:growth}): if the infimum of $u$
vanishes on every horizontal hyperplane, then $u$ has to grow at least like
$x_n^{1-\eps}$ along suitable vertical segments, and this growth, set against the a
priori bound, excludes a broad class of solutions (Theorem~\ref{thm:case2}). What
stays open afterwards is a single, sharply located case, described in
Remark~\ref{rem:open}.

The paper is arranged as follows. Section~1 fixes the hypotheses and states the main
results. Section~2 proves the a priori estimate and works out the model case together
with the mean-curvature and exponential examples. Section~3 proves the first two
nonexistence theorems and the corollary, and includes the reduction lemma that
disposes of the case $p\ge n$ outright. Section~4 contains the growth lemma and the
final theorem.

\section{Formulation of Problems and Main Results}

In this paper, following the framework of D'Ambrosio--Mitidieri \cite{DM0}, we consider nonlinear elliptic inequalities of the form
\begin{equation}\label{eq:0}
-\dv {\cal A}(x,u(x),\nabla u(x)) \ge f(u(x)), \quad \bigl(x=(x',x_n)\in \Rp=\{x_n>0\}\bigr),
\end{equation}
where $\cal A$ and $f$ belong to the classes introduced there (see hypotheses below).

\begin{center}
{\bf Hypotheses}
\end{center}

\begin{itemize}
\item[$({\cal A})$] Let $p>1$ and $p'=\tfrac{p}{p-1}$. The map ${\cal A}:\Rp \times [0,\infty) \times \R^n \to \R^n$ is a Carath\'eodory function
which is {\it weakly-$p$-coercive} (W-$p$-C) in the sense of \cite[Section 2]{DM0}: there exists $k>0$ such that
\begin{equation}\label{eq:1}
{\cal A}(x,t,\xi)\cdot\xi\ge k|{\cal A}(x,t,\xi)|^{p'}\mbox{ for all }(x,t,\xi)\in \Rp \times [0,\infty) \times \R^n
\end{equation}

\item[$(f)$] $f:[0,\infty)\to [0,\infty)$ is continuous, nondecreasing, and $\lim\limits_{t\to +\infty}f(t)=+\infty$. Its generalized
(right-continuous) inverse is
\begin{equation}\label{eq:2}
f^{-1}(s):=\sup\limits_{t\ge 0:f(t)\le s}t, s\ge 0,
\end{equation}
which is finite for every $s\ge 0$ precisely because $f(t)\to +\infty$.
\end{itemize}

As the model case, we can consider ${\cal A}(x,t,\xi)=|\xi|^{p-2}\xi$,
which corresponds to the $p$-Laplace operator $\Delta_p u=\dv(|\nabla u|^{p-2}\nabla u)$. Another interesting example is ${\cal A}(\xi)=\tfrac{\xi}{\sqrt{1+|\xi|^2}}$, which gives rise to the mean curvature operator $\dv\left(\tfrac{\nabla u}{\sqrt{1+|\nabla u|2}}\right)$. For the function $f$, we can take $f(t)=t^q$ with $q>0$ or $f(t)=e^t$. Our main result is as follows.

\begin{theorem}[a priori $L^\infty$ estimate]\label{corr:thm1}
Assume the hypotheses above. Let $u\in C^1(\Rp)\cap C(\overline{\Rp})$ be a nonnegative distributional
supersolution of $-\dv{\cal A}(x,u,\nabla u)\ge f(u)$ in $\Rp$, i.e.
\[
\int_{\Rp}{\cal A}(x,u,\nabla u)\cdot\nabla\psi\,dx\ \ge\ \int_{\Rp} f(u)\,\psi\,dx
\qquad\forall\,\psi\in C^\infty_0(\Rp),\ \psi\ge0,
\]
with $c_1:=\|\nabla u\|_{L^\infty(\Rp)}<\infty$ and $c_2:=\|u\|_{C(\pt\Rp)}<\infty$. Put $c_0:=2^n-1$.
Then for every $R>0$,
\begin{equation}\label{corr:eq5}
\|u\|_{C(\Rp)}\ \le\ \max\!\left(\,c_2+c_1R,\ \ f^{-1}\!\Big(\tfrac{c_0\,k^{1-p}c_1^{p-1}}{R}\Big)+2c_1R\right).
\end{equation}
If moreover $f$ is strictly increasing with $f(0)=0$ and $c_1>0$, then
\begin{equation}\label{corr:eq6}
R\,f(c_1R)=c_0\,k^{1-p}c_1^{p-1}
\end{equation}
has a unique root $R_*\in(0,\infty)$, and taking $R=R_*$ in \eqref{corr:eq5} gives
\begin{equation}\label{corr:eq7}
\|u\|_{C(\Rp)}\ \le\ \max\!\left(c_2+c_1R_*,\ 3c_1R_*\right).
\end{equation}
\end{theorem}

The proof of this theorem and concrete results for the special cases listed above can be found in Section 2. Here we formulate only a corollary for the model case that we will need in the sequel.

\begin{corollary}[model case]\label{corr:cor1}
Let ${\cal A}(x,t,\xi)=|\xi|^{p-2}\xi$ (so $\dv{\cal A}(x,u,\nabla u)=\Delta_p u$) and $f(t)=t^q$ with $q>p-1$. Then ${\cal A}$
is weakly $p$-coercive with $k=1$, since ${\cal A}(\xi)\cdot\xi=|\xi|^p=|{\cal A}(\xi)|^{p'}$. Equation \eqref{corr:eq6}
reads $R\,(c_1R)^q=c_0\,c_1^{p-1}$, i.e.\ $R_*^{\,q+1}=c_0\,c_1^{\,p-1-q}$, whence
\[
R_*=c_0^{\frac{1}{q+1}}c_1^{-\frac{q-p+1}{q+1}},\qquad c_1R_*=c_0^{\frac{1}{q+1}}c_1^{\frac{p}{q+1}},
\]
and \eqref{corr:eq7} becomes
\begin{equation}\label{corr:eq7b}
\|u\|_{C(\Rp)}\ \le\ \max\!\left(c_2+(2^n-1)^{\frac{1}{q+1}}c_1^{\frac{p}{q+1}},\ \
3\,(2^n-1)^{\frac{1}{q+1}}c_1^{\frac{p}{q+1}}\right).
\end{equation}
The exponent $\tfrac{p}{q+1}$ is the one dictated by the scaling $u_\lambda(x)=\lambda^{\frac{p}{q-p+1}}u(\lambda x)$.
\end{corollary}

The bound $|\nabla u|\in C(\Rp)$ restricts the class of solutions we treat, and we prefer to
be plain about it. That condition is natural in the situation that
motivates our research: when a special case of \eqref{eq:0} considered below is obtained by rescaling a problem on a smooth bounded
domain near a boundary point, the rescaled solutions carry a uniform Lipschitz bound. Two consequences of the bound will be used. Since $u=0$ on $\pt\Rp$,
\begin{equation}\label{eq:linbound}
0\le u(x',x_n)\le c_1\,x_n,
\end{equation}
so a solution grows at most linearly in $x_n$; and the same bound controls the oscillation of $u$
on balls. Combining this with the key observation of \cite{DM0} that weak
$p$-coercivity \eqref{eq:1} self-improves to a pointwise bound on $\cal A$ evaluated along the solution, we obtain estimates \eqref{corr:eq5} and \eqref{corr:eq7}.

In Sections 3 and 4, we use the bounds of Theorem \eqref{corr:thm1} to obtain new nonexistence results for the quasilinear Dirichlet problem
\begin{equation}\label{eq:1a}
\begin{cases}
(-\Delta_p u)(x) = f(u(x)) & \bigl(x=(x',x_n)\in \Rp\bigr),\\[2pt]
u(x)=0 & (x\in\pt\Rp),
\end{cases}
\end{equation}
where the operator $\Delta_p$ is defined as above, $n\ge 2$, $p>1$, and $f$ satisfies:

\medskip
\noindent$(h_f)$ $f$ is strictly positive and locally Lipschitz continuous on
$\R_+$, continuous at $0$,
\begin{equation}\label{eq:2a}
\exists \ell=\lim_{t\to 0^+}\frac{f(t)}{t^{p-1}} \in [0, +\infty),
\end{equation}
\medskip

and there are constants $c>0$ and $q>p-1$ with
\begin{equation}\label{eq:3a}
f(t)\ge c\,t^{q}\qquad\forall\, t>0.
\end{equation}
The model case is $f(t)=t^q$ with $q>p-1$.

We look for $u\in W^{1,p}_{loc}(\Rp)\cap L^{q}_{loc}(\Rp)$ that satisfy \eqref{eq:1a} in the
distributional sense and have an essentially bounded gradient, $|\nabla u|\in L^\infty(\Rp)$.
It is known that under our assumptions such solutions are in fact $C^{1,\alpha}_{loc}(\overline{\Rp})$
with some $\alpha>0$, therefore $|\nabla u|\in C(\Rp)$ similarly to Section 2.

Condition $(h_f)$ and the class of solutions under consideration are taken from \cite{FMS},\cite{FMRS}
(in case $1<p<2$, where \eqref{eq:2} follows from $f(0)=0$ and the Lipschitz continuity of $f$ at 0) and \cite{FMS2} (in case $p\ge 2$,
where \eqref{eq:2} is assumed explicitly).

Problem \eqref{eq:1a} is of interest in itself and also through its connection with solvability of
similar equations in smooth bounded domains, which reduce to \eqref{eq:1a} by rescaling near
boundary points. This theory bears on the modeling of nonlinear diffusion, non-Newtonian fluid
flows, and related processes.

In the semilinear case $p=2$, nonexistence for \eqref{eq:1a} in various classes and ranges of $q$
was obtained in \cite{GS81,Dancer}, and for every $q>1$ in \cite{ChenLinZou,DSS}; a wider class of
semilinear problems was treated recently in \cite{DM1,DM2}. The case $1<p<2$ was studied in
\cite{FMRS,FMS} and $p>2$ in \cite{FMS2}, as well as by the present authors in \cite{GS16,Sal}.
It is known that such solutions $u(x)=u(x',x_n)$, when they exist, are strictly monotone in $x_n$ for each fixed
$x'\in\R^{n-1}$, with $\pt u_{x_n}>0$, and hence belong to $C^{2,\beta}_{loc}(\Rp)$ with some $\beta>0$.
In \cite{FMRS,FMS,FMS2} exponents $q_c=q_c(n,p)$ were found for which \eqref{eq:1}
with $f(u)=u^q$ has no solution in these classes when $p-1<q<q_c$:
\[
q_c(n,p)=
\begin{cases}
\infty & \text{for } n\le \dfrac{p(p+3)}{p-1},\\[8pt]
\dfrac{[(p-1)n-p]^2+p^2\bigl[p-2-(p-1)n+2\sqrt{(p-1)(n-1)}\bigr]}
      {(n-p)\bigl[(p-1)n-p(p+3)\bigr]} & \text{for } n>\dfrac{p(p+3)}{p-1}
\end{cases}
\]
and no bounded solution in the same classes for $p-1<q<q_c(n-1,p)$.
However, these estimates are governed by $q_c(n,p)$ or $q_c(n-1,p)$ and say nothing once $q\ge q_c$, which is the range we address. For $p=2$
the full nonexistence is known for every $q>1$ \cite{DSS}, so the case of interest is $p\ne2$.

First of all, by Corollary~\ref{corr:cor1} with $c_2=0$ (the Dirichlet condition kills the boundary term), any weak nonnegative solution
of \eqref{eq:1a} with a globally bounded gradient satisfies
\begin{equation}\label{eq:apriori0}
\sup_{\Rp} u\ \le\ M,\qquad
M:=3\,(2^n-1)^{\frac{1}{q+1}}\,c\,^{-\frac{1}{q+1}}\,c_1^{\frac{p}{q+1}} ,
\end{equation}
where $c>0$ is the constant of $f(t)\ge c\,t^q$ in \eqref{eq:3a}. Combining this estimate with the nonexistence results of
\cite{FMRS,FMS,FMS2} for bounded solutions, we immediately extend the nonexistence range for any solutions in the same class
to $p-1<q<q_c(n-1,p)$:

\begin{theorem}\label{thm2}
Under $(h_f)$ and \eqref{eq:3a}, problem \eqref{eq:1a} has no non-negative nontrivial
solution $u\in W^{1,p}_{loc}(\Rp)\cap L^{q}_{loc}(\Rp)$ with $|\nabla u|\in L^\infty(\Rp)$
 for any $p>1$ and $p-1<q<q_c(n-1,p)$.
\end{theorem}

This result is new for $q_c(n,p)\le q<q_c(n-1,p)$ without assuming that $u$ is bounded.

Our next result on problem \eqref{eq:1a} is in fact independent of Theorem \ref{corr:thm1}. Rather, it is based on the capacity estimates of \cite{MP,BVP} and the monotonicity results of \cite{FMRS,FMS,FMS2}. It rules out a solution that stays bounded below along some hyperplane for $q\ge q_c$.

\begin{theorem}\label{thm:case1}
Under $(h_f)$ and \eqref{eq:3a}, for any $p>1$ and $q>p-1$, problem \eqref{eq:1a} has no non-negative nontrivial
solution $u\in W^{1,p}_{loc}(\Rp)\cap L^{q}_{loc}(\Rp)$ with $|\nabla u|\in L^\infty(\Rp)$ for which
\begin{equation}\label{eq:H1}
\inf_{x'\in\R^{n-1}} u(x',y_0)>0 \quad\text{for some } y_0>0.
\end{equation}
Equivalently, every such solution satisfies $\inf_{x'} u(x',y)=0$ for all $y>0$.
\end{theorem}

As a corollary, we obtain an unconditional nonexistence result for the ``more coercive'' case where the limit in \eqref{eq:2a} is strictly positive (in particular, for $f(t)=c_1 t^{p-1}+c_2 t^q$ with $c_1>0$, $c_2>0$, $q>p-1$).

\begin{corollary}\label{cor2}
Assume $(h_f)$ with $p\ge 2$ and a strictly positive limit in \eqref{eq:2a}:
\begin{equation}\label{eq:2aa}
\ell=\lim_{t\to 0^+}\frac{f(t)}{t^{p-1}} \in (0, +\infty),
\end{equation}
and let \eqref{eq:3a} hold. Then, for any $p>1$ and $q>p-1$, problem \eqref{eq:1a} has no non-negative nontrivial
solution $u\in W^{1,p}_{loc}(\Rp)\cap L^{q}_{loc}(\Rp)$ with $|\nabla u|\in L^\infty(\Rp)$.
\end{corollary}

To formulate our final result, we need to introduce some classes of sequences $(x'_k,x_{k,n})$.

\begin{definition}\label{nonadmiss}
Fix $\eps\in(0,1)$ and a base height $\delta\in(0,1]$. We define three classes of sequences
$x_k=(x'_k,x_{k,n})\in\Rp$ with $x_{k,n}\to+\infty$ by the following properties:
\begin{itemize}
\item[(i)] $\displaystyle\lim_{k\to\infty}u(x'_k,x_{k,n})=0$;
\item[(ii)] $\displaystyle\frac{\pt u(x'_k,x_n)}{\pt x_n}\ \ge\ \frac{1-\eps}{x_n}\,u(x'_k,x_n)$
\quad for all $x_n\in[\delta,x_{k,n}]$;
\item[(iii)] $\displaystyle\limsup_{k\to\infty}\, u(x'_k,\delta)\,x_{k,n}^{\,1-\eps}\ >\ M,$ where $M$ is the apriori bound from (\ref{eq:apriori0}).
\end{itemize}
\end{definition}

Note that existence of sequences with property (ii) is not an extra assumption on $u$: under the zero-limit hypothesis
\begin{equation}\label{eq:2b}
\ell=\lim_{t\to 0^+}\frac{f(t)}{t^{p-1}}=0,
\end{equation}
Lemma~\ref{lem:growth} constructs, from any sequence with property (i), points along which (ii) holds.

\begin{theorem}\label{thm:case2}
Assume $(h_f)$ with the zero limit \eqref{eq:2b}, and \eqref{eq:3a}. Let $p>1$, $q>p-1$. Then
problem \eqref{eq:1a} has no non-negative nontrivial solution
$u\in W^{1,p}_{loc}(\Rp)\cap L^q_{loc}(\Rp)$ with $|\nabla u|\in L^\infty(\Rp)$, for which any 
sequence with properties (i) and (ii) has property (iii).
\end{theorem}

The technical step behind Theorem~\ref{thm:case2}, and the reason for the exponent $1-\eps$, is a
growth estimate of independent interest (Lemma~\ref{lem:growth}): whenever the infimum of $u$
vanishes on every hyperplane $\{x_n=c>0\}$, $u$ grows at least like $x_n^{1-\eps}$ along suitable vertical
segments of the form $\{(x'_k,x_n):\,\delta\le x_n\le x_{k,n}$, so that the sequence $\{x'_k,x_{k,n}\}$
has at least properties (i) and (ii); thus, if at least one such sequence has property (iii) as well,
we obtain a contradiction with the apriori bound \eqref{corr:eq7}. The three last theorems together leave open only the solutions
of \eqref{eq:1a} with $q\ge q_c(n-1,p)$, for which the infimum vanishes on every such hyperplane and no sequence of vertical lines
carries the growth
\begin{equation}\label{eq:H2}
u(x'_k,x_n)\ \ge\ \delta\,x_n^{1-\eps}\qquad(\delta\le x_n\le x_{k,n});
\end{equation}
see Remark~\ref{rem:open}.

\section{Proof of Theorem~\ref{corr:thm1} and Corollary~\ref{corr:cor1}. Examples}

\begin{proof}[Proof of Theorem~\ref{corr:thm1}]
\textbf{Step 1 (pointwise flux bound).} For a.e.\ $x$ set $a:=|{\cal A}(x,u,\nabla u)|$. By \eqref{eq:1} and
Cauchy--Schwarz,
\[
k\,a^{p'}\ \le\ {\cal A}(x,u,\nabla u)\cdot\nabla u\ \le\ a\,|\nabla u|\ \le\ a\,c_1 .
\]
If $a>0$, divide by $ka$: $\,a^{\,p'-1}\le c_1/k$. Since $p'-1=\tfrac{1}{p-1}$, raising to the power $p-1$
gives
\begin{equation}\label{corr:flux}
|{\cal A}(x,u,\nabla u)|\ \le\ (c_1/k)^{p-1}=k^{1-p}c_1^{p-1}\qquad\text{a.e. in }\Rp,
\end{equation}
trivially also when $a=0$. \emph{(This is the D'Ambrosio--Mitidieri self-improvement: no upper growth bound
on ${\cal A}$ is used.)}

\textbf{Step 2 (capacity estimate; infimum bound).} Fix $x_*=(x'_*,x_{*,n})$ with $B_{2R}(x_*)\subset\Rp$ and
the radial Lipschitz cutoff $\varphi(x)=\min\!\big(1,\tfrac{1}{R}(2R-|x-x_*|)_+\big)$, so $0\le\varphi\le1$,
$\varphi\equiv1$ on $B_R(x_*)$, $\varphi\equiv0$ off $B_{2R}(x_*)$, and $|\nabla\varphi|=\tfrac1R$ a.e.\ on the
annulus $A_R=B_{2R}(x_*)\setminus B_R(x_*)$, with $|A_R|=(2^n-1)\omega_n R^n$. Test functions may be taken
Lipschitz with compact support by mollification; the integrands converge boundedly by \eqref{corr:flux} and
continuity of $f(u)$. Hence, using the supersolution inequality, then \eqref{corr:flux}, then $|A_R|$,
\begin{equation}\label{corr:eq9}
\int_{B_{2R}(x_*)} f(u)\,\varphi\,dx
\ \le\ \int_{A_R}{\cal A}(x,u,\nabla u)\cdot\nabla\varphi\,dx
\ \le\ k^{1-p}c_1^{p-1}\cdot\tfrac1R\cdot(2^n-1)\omega_n R^n
\ =\ c_0\,k^{1-p}c_1^{p-1}\,\omega_n R^{\,n-1}.
\end{equation}
On the other hand, $\varphi\equiv1$ on $B_R(x_*)$, $f(u)\ge0$, and $f$ nondecreasing give
\begin{equation}\label{corr:eq10}
\int_{B_{2R}(x_*)} f(u)\,\varphi\,dx\ \ge\ \int_{B_R(x_*)} f(u)\,dx\ \ge\ f\!\big(\inf_{B_R(x_*)}u\big)\,\omega_n R^n.
\end{equation}
Dividing \eqref{corr:eq9}--\eqref{corr:eq10} by $\omega_n R^n>0$,
\[
f\!\big(\inf_{B_R(x_*)}u\big)\ \le\ \frac{c_0\,k^{1-p}c_1^{p-1}}{R},
\qquad\text{hence}\qquad
\inf_{B_R(x_*)}u\ \le\ f^{-1}\!\Big(\tfrac{c_0\,k^{1-p}c_1^{p-1}}{R}\Big)
\]
by the definition of the generalized inverse.

\textbf{Step 3 (infimum $\to$ supremum).} As $u\in C^1(B_R(x_*))$ with $|\nabla u|\le c_1$ and $B_R(x_*)$ is
convex, for $x,y\in B_R(x_*)$ the fundamental theorem of calculus gives $u(x)-u(y)\le c_1|x-y|\le 2c_1R$.
Taking $y$ near the infimum,
\[
\sup_{B_R(x_*)}u\ \le\ \inf_{B_R(x_*)}u+2c_1R\ \le\ f^{-1}\!\Big(\tfrac{c_0\,k^{1-p}c_1^{p-1}}{R}\Big)+2c_1R,
\]
independently of $x_*$.

\textbf{Step 4 (covering and boundary strip).} The centres with $B_{2R}(x_*)\subset\Rp$ are exactly
$\{x_{*,n}\ge2R\}$, and $\bigcup_{x_{*,n}\ge2R}B_R(x_*)=\{x_n>R\}$. Thus
\[
\sup_{\{x_n>R\}}u\ \le\ f^{-1}\!\Big(\tfrac{c_0\,k^{1-p}c_1^{p-1}}{R}\Big)+2c_1R .
\]
On the strip $\{0<x_n\le R\}$, the $c_1$-Lipschitz bound (extended to $\overline{\Rp}$ by continuity) compared
with the boundary value gives $\sup_{\{0<x_n\le R\}}u\le c_2+c_1R$. Taking the maximum proves
\eqref{corr:eq5}; $R>0$ was arbitrary.

\textbf{Step 5 (balancing).} If $f$ is strictly increasing with $f(0)=0$, $f(+\infty)=+\infty$ and $c_1>0$,
then $g(R):=R\,f(c_1R)$ is continuous, strictly increasing, $g(0^+)=0$, $g(+\infty)=+\infty$; so
\eqref{corr:eq6} has a unique root $R_*>0$. There $f(c_1R_*)=c_0k^{1-p}c_1^{p-1}/R_*$, and as $f^{-1}$ is now
the ordinary inverse, $f^{-1}\!\big(c_0k^{1-p}c_1^{p-1}/R_*\big)=c_1R_*$. Substituting $R=R_*$ into
\eqref{corr:eq5} yields \eqref{corr:eq7}. (If $c_1=0$, then $u$ is constant; letting $R\to\infty$ in
\eqref{corr:eq10} forces $f(u)\equiv0$ and \eqref{corr:eq5} reduces to $\|u\|\le c_2$.)
\end{proof}

\begin{remark}\label{rem:bound}
Only the pointwise flux bound \eqref{corr:flux} enters; no upper growth or extra regularity of ${\cal A}$ in
$(t,\xi)$ is used beyond the Carath\'eodory condition. The bound is cruder than the Mitidieri--Pohozaev
integral estimate $\int_{B_R}u^q\le cR^{\,n-\frac{pq}{q-p+1}}$, yet after balancing it yields the same sharp
exponent $\tfrac{p}{q+1}$.
\end{remark}

\begin{example}\label{ex1}
(Mean curvature operator). ${\cal A}(\xi)=\tfrac{\xi}{\sqrt{1+|\xi|^2}}$ is W-$2$-C with $k=1$, since
\[
|{\cal A}(\xi)|^2=\tfrac{|\xi|^2}{1+|\xi|^2}\le \tfrac{|\xi|^2}{\sqrt{1+|\xi|^2}}={\cal A}(\xi)\cdot\xi
\]
(cf. \cite{DM0}). Thus Theorem \ref{corr:thm1} applies to nonnegative supersolutions of $\dv\left(\frac{\nabla u}{\sqrt{1+|\nabla u|^2}}\right)\ge f(u)$
in $\Rp$ with $p=2, k=1$, estimate \eqref{corr:eq7} holds with $R_*$ defined by $Rf(c_1 R)=(2^n-1)c_1$.
For $f(t)=t^q\,(q>1)$ this gives $\|u\|_{C(\Rp)}\le c_2+(2^n-1)^{\frac{1}{q+1}}c_1^{\frac{2}{q+1}}$ up to the maximum in \eqref{corr:eq7}.
\end{example}

\begin{example}\label{ex2}
(Exponential nonlinearity). For $f(t)=e^t-1$ (strictly increasing, f(0)=0) and ${\cal A}$
W-$p$-C, \eqref{corr:eq5} with, e.g.,$R=1$ gives the bound
\[
\|u\|_{C(\Rp)}\le \max\left(c_2+c_1, \log(1+(2^n-1)k^{1-p}c_1^{p-1})+2c_1\right),
\]
and optimizing over $R$ via \eqref{corr:eq6} refines it further: for small $c_1$ the bound behaves like $c_1 \log\frac{1}{c_1}$-type
quantities. This illustrates that no polynomial structure of the nonlinearity is needed.
\end{example}

\begin{remark}\label{rem1}
Two features are inherited from \cite{DM0}: (i) no upper growth bound on $\cal A$ is assumed --
 the W-$p$-C condition \eqref{eq:1} alone produces the pointwise flux bound \eqref{corr:flux} along the solution; (ii)
no regularity of $\cal A$ in $(t,\xi)$ beyond the Carath\'eodory condition is used, since the proof involves
a single integration by parts against a cutoff. Compared with the Mitidieri--Pohozaev capacity
estimate \cite{MP, BVP}, the present argument replaces the sharp integral bound $\int_{B_R}
u^q\,dx \le cR^{n-\frac{pq}{q-p+1}}$ by the cruder \eqref{corr:eq9}; nevertheless, after balancing $R$ the resulting exponent
$\frac{p}{q+1}$ in Corollary \eqref{corr:cor1} is the same as it must be, being dictated by the scaling $u_{\lambda}(x)=\lambda^{\frac{p}{q-p+1}}u(\lambda x)$.
\end{remark}

\section{Proof of Theorems~\ref{thm2}, \ref{thm:case1} and Corollary~\ref{cor2}}

\begin{proof}[Proof of Theorem~\ref{thm2}] By \eqref{corr:eq7} with $c_1=\|\nabla u\|_{\infty}$ and $c_2=0$ (due to Dirichlet boundary condition), we have
\[
\|u\|_{C(\Rp)}\le 3(2^n-1)^{\frac{1}{q+1}}\|\nabla u\|_{\infty}^{\frac{p}{q+1}},
\]
and the claim immediately follows from nonexistence results for bounded solutions, namely, from the second parts of \cite[Thm.~1.6]{FMS}, \cite[Thm.~1.4]{FMRS} and \cite[Thm.~1.3]{FMS2}.
\end{proof}

\noindent\textbf{Reduction.}
By \cite[Thm.~4.4 (i) and Rem.~4.2]{BVP} with $\sigma=0$, the inequality
\begin{equation}\label{eq:4a}
(-\Delta_p v)(x)\ge v^{s}(x)
\end{equation}
has no non-negative nontrivial supersolution on a half-space, provided $1<p\le n$ and $p-1<s<b_p$,
where $b_p=p-1+\tfrac{p}{\beta_{p,n}}>p-1$ with some constant
$\beta_{p,n}>0$. The explicit values of $\beta_{p,n}$ are given only for $n=2$
and $p=2$, but they are not required for our argument; in particular, for monotone solutions one can take
$b_p=s_p:=\tfrac{n(p-1)}{n-p}$ if $1<p<n$ and $s_p=+\infty$ if $p\ge n$ (see the lemma below). The statement is
about the inequality alone and needs no boundary condition, since the underlying Mitidieri–Pohozaev test-function
estimate
\begin{equation}\label{eq:apriori}
\int_{B_R(x_*)} v^{\gamma}(x)\,dx\ \le\ c\,R^{\,n-\frac{p\gamma}{s-p+1}}
\qquad(\text{see \cite{MP} and \cite[proof of Theorem 3.3]{BVP}})
\end{equation}
holds for any $0\le\gamma\le q$, $R>0$ and $x_*$ such that $B_{2R}(x_*)\subset\Rp$.
This covers every shifted half-space. In particular we may assume $q\ge s_p$ in
\eqref{eq:1a}.

\begin{lemma}\label{lem:reduction}
Let $v \ge 0$ be a distributional supersolution of \eqref{eq:4a}
on a shifted half-space, nondecreasing in $x_n$; if $p-1<s<s_p$, with $s_p=\tfrac{n(p-1)}{n-p}$
for $p<n$ and $s_p=+\infty$ for $p\ge n$, then $v\equiv 0$.
\end{lemma}

\begin{proof}
Under the assumptions of the lemma, the estimate \eqref{eq:apriori} with $\gamma=s=q$:
\begin{equation}\label{eq:apriori1}
\int_{B_R(x_*)} v^q(x)\,dx\ \le\ c\,R^{\,n-\frac{pq}{q-p+1}}
\end{equation}
has a negative exponent, since the inequality $n(q-p+1)<pq$ is equivalent
to $q(n-p)<n(p-1)$, which is automatic when $n\le p$ and results in $s<s_p$ when $n>p$. Applying the estimate \eqref{eq:apriori1} on balls of
arbitrarily large radius contained in the half-space and using monotonicity, for the integrals
$\int_{C_R} v^q(x)\,dx$, where $C_R=\{(x',x_n):|x'|<R, 0<x_n<3R\}$, we have
\[
\int_{C_R} v^q(x)\,dx \le \int_{{C_R}+(0',3R)} v^q(x)\,dx \le\int_{B_{\rho}((0',9R/2))} v^q(x)\,dx\,
\]
where
\[
{C_R}+(0',3R)=\{(x',x_n):|x'|<R, 3R<x_n<6R\} \subset B_{\rho}((0',9R/2)), \quad \rho=\tfrac{\sqrt{13}R}{2},
\]
and the ball $B_{\rho}((0',9R/2))$ is admissible: $B_{2\rho}((0',9R/2))=B_{\sqrt{13}}((0',9R/2))\subset\Rp$,
because $9/2>\sqrt{13}{\cal A}pprox 3,61$, so the estimate \eqref{eq:apriori1} with radius $\rho \sim R$ gives
\[
\int_{C_R} v^q(x)\,dx \le c'\,R^{\,n-\frac{pq}{q-p+1}}\to 0,
\]
whence $v \equiv 0$ in the whole half-space.
\end{proof}

Hence problem \eqref{eq:1a} has no nontrivial
nonnegative solution whatsoever when $p\ge n$, with no hypothesis beyond \eqref{eq:3a}. Thus
we may assume $1<p\le n$ and $q\ge s_p=\tfrac{n(p-1)}{n-p}$ throughout.

We also record the regularity used below. Solutions of the $p$-Laplace equation are in general
only $C^{1,\alpha}_{loc}$, and $C^{2,\beta}_{loc}$ regularity needs the gradient locally bounded
away from zero. Here this holds because $u$ is strictly monotone in $x_n$, so $\pt_{x_n}u>0$ and
$\nabla u\ne 0$ in $\Rp$.

\bigskip
\begin{proof}[Proof of Theorem~\ref{thm:case1}]
Suppose a non-negative nontrivial solution $u$ with \eqref{eq:H1} exists. By the monotonicity of
$u(x',\cdot)$ and by \eqref{eq:3a},
\begin{equation}\label{eq:6a}
(-\Delta_p u)(x)\ge c_1 u^{s}(x)\qquad \bigl(x\in\R^{n-1}\times(y_0,+\infty)\bigr)
\end{equation}
holds for any $s<q$, in particular for some $s\in(p-1,b_p)$, with
$c_1=c\,\inf_{x'} u^{\,q-s}(x',y_0)>0$ by \eqref{eq:H1}. Inequality \eqref{eq:6a} lives on the
shifted half-space $\R^{n-1}\times(y_0,+\infty)$, and $u$ does not vanish on $\{x_n=y_0\}$; it is
bounded below there. This causes no difficulty, because the reduction concerns the
inequality on a shifted half-space and imposes no boundary condition. Setting
$v=c_1^{1/(s-p+1)}u$ and translating $x_n\mapsto x_n-y_0$ turns \eqref{eq:6a} into \eqref{eq:4a} on
$\Rp$, which is impossible due to Lemma~\ref{lem:reduction}. Hence no such solution exists.
\end{proof}

\begin{proof}[Proof of Corollary~\ref{cor2}]
By \cite[Lemma~18]{Le}, for a solution $u\in C^{1,\alpha}_{\rm loc} (\Rp) \cap C(\overline{\Rp})$
of problem \eqref{eq:1a} with $f(t)>c_0 t^{p-1}, t\in (0,t_0), c_0>0, t_0>0$ (which follows in our case from assumption \eqref{eq:2a};
due to \eqref{eq:3a}, one can take arbitrary $t_0>0$), and for each $x=(x',x_n)\in\Rp$, one has $u(x)\ge \min (Cx_n,t_0)$,
which implies \eqref{eq:H1}. Then the claim follows immediately from Theorem~\ref{thm:case1}.
\end{proof}

\section{Proof of Theorem~\ref{thm:case2}}

\begin{lemma}\label{lem:growth}
Assume $(h_f)$ with a zero limit in \eqref{eq:2a} (that is, \eqref{eq:2b}) and \eqref{eq:3a}. Let $u$ be a solution of \eqref{eq:1a}
in the class above satisfying
\begin{equation}\label{eq:7a}
\inf_{x'\in\R^{n-1}} u(x',y)=0\qquad\text{for every } y>0 .
\end{equation}
Then for every $k\in\N$ and $\eps\in (0,1)$ there is a point $x'_{(k)}\in\R^{n-1}$ with
\begin{equation}\label{eq:10a}
u(x'_{(k)},x_n)>x_n^{1-\eps}\,u(x'_{(k)},1)/2 \qquad (x_n\in[1,(1-\eps)^{-k/2}]).
\end{equation}
\end{lemma}

\begin{proof}
Fix $k\in\N$ and put $y_j=\delta(1-\eps)^{-j/2}$, $j=0,1,\dots,k$. By \eqref{eq:7a} there is a sequence
$\{x'_m\}\subset\R^{n-1}$ with $u(x'_m,y_k)\to 0$ as $m\to\infty$. Since $u\ge 0$ and $u(x'_m,\cdot)$
is increasing, $0\le u(x'_m,y_j)\le u(x'_m,y_k)$ for $j\le k$, so $u(x'_m,y_j)\to 0$ for every $j$.
Write
\[
v_{m,j}(x)=\frac{u(x'+x'_m,x_n)}{u(x'_m,y_j)},
\]
which satisfies
\begin{equation}\label{eq:8a}
(-\Delta_p v_{m,j})(x)=c_{m,j}(x)\,v_{m,j}^{\,p-1}(x),\qquad
c_{m,j}(x)=\frac{f\bigl(u(x'+x'_m,x_n)\bigr)}{u(x'+x'_m,x_n)^{\,p-1}} .
\end{equation}
Writing $u(x'+x'_m,x_n)=u(x'_m,y_j)\,v_{m,j}(x)$, condition \eqref{eq:2a} shows that $c_{m,j}$ is
locally uniformly bounded and, because $u(x'_m,y_j)\to0$, that $c_{m,j}\to0$ locally uniformly on
any region where $v_{m,j}$ stays bounded. Also $v_{m,j}(0,y_j)=1$.

By the Harnack inequality \cite[Thm.~7.2.2]{PS}, for any compact $D\subset\Rp$ with $(0,y_j)\in D$
and any $\theta>0$,
\[
\sup_{D\cap\{x_n>\theta\}} v_{m,j}\ \le\ C_D\inf_{D\cap\{x_n>\theta\}} v_{m,j}\ \le\ C_D,
\]
and the monotonicity of $u(x',\cdot)$ upgrades this to
\begin{equation}\label{eq:10b}
\sup_D v_{m,j}\le C_D.
\end{equation}
So the $v_{m,j}$ are locally bounded in $C(D)$ and hence, by standard local bounds,
in $C^{1,\alpha}(D)$ for each $j$. A diagonal argument yields
$v_j=\lim_{m\to\infty}v_{m,j}$ in $C^{1,\alpha'}_{loc}(\Rp)$, $\alpha'\in(0,\alpha)$, along a
subsequence. By \eqref{eq:2b}, \eqref{eq:8a} and $u(x'_m,y_j)\to0$, the limit satisfies
\[
-\Delta_p v_j=0 \text{ in } \Rp,\qquad v_j=0 \text{ on } \pt\Rp,\qquad v_j(0,y_j)=1,
\]
and $v_j>0$ in $\Rp$ by the strong maximum principle \cite{Vazquez}. Moreover, since $c_{m,j}$ is
bounded up to the boundary uniformly in $m$ (by $u(x) \le Cx_n$ and the finite limit in \eqref{eq:2a}),
estimate \eqref{eq:10b} can be extended to $D\subset\overline{\Rp}$, so Lieberman's boundary $C^{1,\alpha}$
estimates \cite{Lieberman} give equicontinuity up to $\{x_n = 0\}$ and the limit $v_j$ vanishes there.

A non-negative $p$-harmonic function on $\Rp$ vanishing on $\pt\Rp$ is linear, $c\,x_n$ with
$c\ge0$. This is precisely \cite[Theorem 3.1]{KSZ}, the $p$-Laplacian counterpart of the half-space Martin representation; it asks for no
growth bound on $v_j$, which is what lets us use it here, the rescaling having destroyed the
gradient bound on $v_{m,j}$. With $v_j(0,y_j)=1$ we get $v_j(x)=x_n/y_j$, so
$\pt_{x_n}v_j\equiv 1/y_j$.

Since $v_{m,j}\to v_j$ in $C^{1,\alpha'}_{loc}$ and $\pt_{x_n}v_j\equiv 1/y_j$ on
$\{0\}\times[y_{j-1},y_j]$, we have $\pt_{x_n}v_{m,j}(0,\cdot)\to 1/y_j$ uniformly on
$[y_{j-1},y_j]$. So there is an index $m_j$ with
\[
\pt_{x_n}v_{m,j}(0,x_n)>\frac{\sqrt{1-\eps}}{y_j}\qquad\text{for all } m\ge m_j,\ x_n\in[y_{j-1},y_j].
\]
Put $M=M(k)=\max_{1\le j\le k} m_j$ and $x'_{(k)}:=x'_M$. As the inequality holds for all $m\ge m_j$, it
holds at $m=M$ for every $j=1,\dots,k$ at once. Hence, for each $j$ and each
$x_n\in[y_{j-1},y_j]$,
\[
\frac{\pt u(x'_M,x_n)}{\pt x_n}
=u(x'_M,y_j)\,\pt_{x_n}v_{M,j}(0,x_n)
\ge \frac{\sqrt{1-\eps} \cdot u(x'_M,y_j)}{y_j}
\]
\[
\ge \frac{\sqrt{1-\eps} \cdot u(x'_M,x_n)}{y_j}
\ge \frac{(1-\eps) u(x'_M,x_n)}{x_n},
\]
using the definition of $v_{M,j}$, the monotonicity of $u(x'_M,\cdot)$, and
\[
y_j=\frac{y_{j-1}}{\sqrt{1-\eps}}\le \frac{x_n}{\sqrt{1-\eps}}.
\]
As $j$ runs over $1,\dots,k$, for every $x_n\in[1,y_k]$
\begin{equation}\label{eq:logder}
\frac{\pt u(x'_M,x_n)}{\pt x_n}\ge\frac{(1-\eps) u(x'_M,x_n)}{x_n}.
\end{equation}
Integrating \eqref{eq:logder} gives $u(x'_M,x_n)\ge x_n^{1-\eps}u(x'_M,\delta)$. Since $M=M(k)$, this coincides
with \eqref{eq:10a} up to notation.
\end{proof}

\begin{proof}[Proof of Theorem~\ref{thm:case2}]
Suppose, for contradiction, that such a solution $u$ and a sequence
$\{x_k\}$ with properties (i)--(iii) exist. Recall $u\in C^1(\Rp)$ and $u>0$ in $\Rp$ (strong maximum principle
\cite{Vazquez}), so $x_n\mapsto u(x'_k,x_n)$ is positive and $C^1$ on $[\delta,x_{k,n}]$.

\emph{Step 1 (integrate (ii) to segment growth).} Fix $k$. Dividing (ii) by
$u(x'_k,x_n)>0$ gives the logarithmic-derivative bound
\[
\frac{\pt}{\pt x_n}\,\log u(x'_k,x_n)\ \ge\ \frac{1-\eps}{x_n},\qquad x_n\in[\delta,x_{k,n}].
\]
Integrating from $\delta$ to $x_{k,n}$,
\[
\log\frac{u(x'_k,x_{k,n})}{u(x'_k,\delta)}\ \ge\ (1-\eps)\,\log\frac{x_{k,n}}{\delta},
\]
that is,
\begin{equation}\label{eq:growth}
u(x'_k,x_{k,n})\ \ge\ \Big(\frac{x_{k,n}}{\delta}\Big)^{1-\eps} u(x'_k,\delta)
\ =\ \delta^{-(1-\eps)}\,u(x'_k,\delta)\,x_{k,n}^{\,1-\eps}.
\end{equation}
This is the only place property (ii) is used, and (i) is what made (ii) available (via
Lemma~\ref{lem:growth}).

\emph{Step 2 (cap by the a priori bound).} Since $(x'_k,x_{k,n})\in\Rp$, the bound
\eqref{eq:apriori0} gives $u(x'_k,x_{k,n})\le M$. Combining with \eqref{eq:growth},
\[
\delta^{-(1-\eps)}\,u(x'_k,\delta)\,x_{k,n}^{\,1-\eps}\ \le\ M,
\qquad\text{i.e.}\qquad
u(x'_k,\delta)\,x_{k,n}^{\,1-\eps}\ \le\ \delta^{\,1-\eps}\,M\ \le\ M,
\]
the last inequality because $\delta\le1$ and $1-\eps>0$. This holds for \emph{every} $k$, so
\begin{equation}\label{cap}
\limsup_{k\to\infty}\, u(x'_k,\delta)\,x_{k,n}^{\,1-\eps}\ \le\ M .
\end{equation}

\emph{Step 3 (contradiction with (iii)).} Property (iii) of the sequence $(x'_k,\delta)\,x_{k,n})$ asserts
\[
\limsup_{k\to\infty} u(x'_k,\delta)\,x_{k,n}^{\,1-\eps}> M,
\]
which is incompatible with
\eqref{cap}. This contradiction shows that no non-negative nontrivial solution in the stated
class admits a sequence with properties (i)--(iii).
\end{proof}

\begin{remark}\label{rem:open}
The theorem excludes only solutions that \emph{admit} a sequence with all three properties (i)--(iii), and this
conditionality is genuine, not an artefact of the write-up. Lemma~\ref{lem:growth} always produces, for
each $k$, a point $x'_{(k)}$ with segment growth on $[\delta,x_{k,n}]$. Since
$u(x'_k,x_{k,n})\le M$, the segment growth (ii) forces $u(x'_k,\delta)\,x_{k,n}^{\,1-\eps} \le
\delta^{1-\eps}M\le M$; equivalently the base value decays at least like $\rho^{-k}$, so property (iii) fails.
To close the problem for $q\ge q_c(n-1,p)$ one would have to produce a
\emph{single} vertical line $\{x'=x'_*\}$ carrying the growth \eqref{eq:growth} with a base value
$u(x'_*,\delta)\ge\delta_0>0$ bounded below, independent of $k$; whether the a priori bound forces
such a line is an open question (for $p=2$ the full nonexistence for every $q>1$ is known
\cite{DSS}).
\end{remark}

\noindent{\bf Acknowledgment.} The authors thank Professor Enzo Mitidieri (University of Trieste) for fruitful discussion.

\end{document}